\documentclass[11pt]{article}
\usepackage[total={6.5in, 9in}]{geometry}
\usepackage{amsmath,amsfonts,mathrsfs,amsthm,amssymb}
\numberwithin{equation}{section}
\usepackage{cases}
\usepackage{latexsym,bm}
\usepackage{indentfirst}
\usepackage{color}
\usepackage{ifpdf}
\usepackage{graphicx}
\usepackage{psfrag}
\usepackage{dsfont}
\usepackage{enumerate}

\usepackage{pgf,tikz,pgfplots}
\usepackage{tikz-3dplot}
\usetikzlibrary{positioning,fit,backgrounds,calc,decorations.pathreplacing}
\usepackage{pifont}
\usepackage[T1]{fontenc}
\usepackage{lmodern}
\usetikzlibrary{arrows.meta,backgrounds,calc}

\definecolor{cOne}{HTML}{2563A8}
\definecolor{cTwo}{HTML}{B04A3A}
\definecolor{cThree}{HTML}{3A7D5D}
\definecolor{cQ}{HTML}{247C84}
\definecolor{ink}{HTML}{20242A}
\definecolor{soft}{HTML}{F6F7F9}
\definecolor{joinA}{HTML}{D6E5F4}
\definecolor{joinT}{HTML}{E3E8D2}

\usepackage{refcount}
\usepackage[linesnumbered,ruled,vlined]{algorithm2e}
\usepackage{subcaption}

\usepackage[
pdfauthor={},
pdftitle={},
pdfstartview=XYZ,
bookmarks=true,
colorlinks=true,
linkcolor=blue,
urlcolor=blue,
citecolor=blue,
linktocpage=true,
hyperindex=true
]{hyperref}

\theoremstyle{definition}
\newtheorem{THM}{\textbf{Theorem}}[section]

\newtheorem{LEM}[THM]{\textbf{Lemma}}
\newtheorem{CLA}{\textbf{Claim}}[section]

\newcommand{\comp}{\operatorname{c}}

\providecommand{\phantomsection}{} \makeatletter \newcommand{\CaseLabel}[1]{%
	\smallskip \noindent\textbf{Case #1.}\phantomsection \def\@currentlabel{#1}\label{case:#1}\ } 
\makeatother

\newcommand{\CC}{\mathcal{C}}

\begin{document}
	\title{Triangle-Free Graphs of Toughness Approaching Two\\Without a 2-Factor}
	\author{
		Songling Shan \thanks{Auburn University, Department of Mathematics and Statistics, Auburn, AL 36849. 
			Email:  	{\tt szs0398@auburn.edu}. Partially supported by NSF  grant 
			DMS-2451895.}
	}
	
	\date{\today}
	\maketitle

\begin{abstract}
	
	By work of Enomoto, Jackson, Katerinis, and Saito from 1985, every
	$2$-tough graph has a $2$-factor, and this toughness bound is best
	possible: for every $\varepsilon>0$, there exist $(2-\varepsilon)$-tough
	graphs with no $2$-factor. It is natural to ask whether the latter
	statement remains true for triangle-free graphs. Bauer, van den Heuvel,
	and Schmeichel conjectured this in 1996. In the same paper, they proposed
	an infinite family of triangle-free graphs with no $2$-factor whose
	toughness they believed approaches $2$, but the required toughness bound
	was not established. In this paper, we confirm their conjecture. For
	every even integer $q\ge 6$, we construct a triangle-free graph $G_q$
	with no $2$-factor and with toughness 
	\[
	\tau(G_q)
	=\frac{2q^2-q-2}{q^2+q}
	=2-\frac{3q+2}{q^2+q}.
	\]
	In particular, $\tau(G_q)\to 2$ as $q\to\infty$, showing that the
	threshold $2$ for the existence of a $2$-factor remains best possible
	even within the class of triangle-free graphs.
\end{abstract}


\section{Introduction}

We consider only simple  graphs. 
Let $G$ be a graph.
Denote by $V(G)$ and  $E(G)$ the vertex set and edge set of $G$,
respectively.  
Let $v\in V(G)$ and $S\subseteq V(G)$. 
Then  $N_G(v)$   denotes the set of neighbors
of $v$ in $G$ and  $N_G(S):=(\bigcup_{v\in S}N_G(v))\setminus S$. 
The subgraph of $G$ induced on $S$ and $V(G)\setminus S$ are denoted by
$G[S]$ and $G-S$, respectively. For notational simplicity we write $G-v$ for $G-\{v\}$.
Let $V_1,
V_2\subseteq V(G)$ be two disjoint vertex sets. Then $E_G(V_1,V_2)$ is the set
of edges in $G$  with one end in $V_1$ and the other end in $V_2$ and  $e_G(V_1,V_2):=|E_G(V_1,V_2)|$.  We write $E_G(v,V_2)$ and $e_G(v,V_2)$
if $V_1=\{v\}$ is a singleton.  
When $H\subseteq G$ and $S\subseteq V(G)\setminus V(H)$, we write $N_G(H)$, $E_G(H,S)$,
and $e_G(H,S)$ respectively for $N_G(V(H))$, $E_G(V(H),S)$,
and $e_G(V(H),S)$.  For two integers $p$ and $q$, we let $[p,q]=\{i\in \mathbb{Z}:  p\le i\le q\}$.

The number of components of $G$ is denoted by $c(G)$. Let $t\ge 0$ be a
real number. The graph $G$ is said to be \emph{$t$-tough} if $|S|\ge t\cdot
c(G-S)$ for each $S\subseteq V(G)$ with $c(G-S)\ge 2$. The \emph{toughness $\tau(G)$} is the largest real number $t$ for which $G$ is
$t$-tough, or is  $\infty$ if $G$ is complete. This concept was introduced by Chv\'atal~\cite{Chvatal1973} in 1973.
  A spanning
$2$-regular subgraph is called a \emph{$2$-factor}.  Enomoto, Jackson,
Katerinis, and Saito proved that every $2$-tough graph has a
$2$-factor~\cite{EnomotoEtAl1985}, and this toughness bound is best
possible: for every $\varepsilon>0$, there exist $(2-\varepsilon)$-tough
graphs with no $2$-factor.  It is natural to ask whether the latter
statement remains true for triangle-free graphs. 
Questions connecting toughness, triangle-freeness, and $2$-factors were
studied in particular by Bauer, van den Heuvel, and
Schmeichel~\cite{BauerVandenHeuvelSchmeichel1995,BauerVandenHeuvelSchmeichel1996}.
In particular, 
Bauer, van den Heuvel,
and Schmeichel conjectured   in 1996 that for every $\varepsilon>0$, there exist $(2-\varepsilon)$-tough
triangle-free graphs with no $2$-factor~\cite{BauerVandenHeuvelSchmeichel1996}. In the same paper, they proposed
an infinite family of triangle-free graphs with no $2$-factor whose
toughness they believed approaches $2$, but the required toughness bound
was not established. In this paper, we confirm their conjecture.

\begin{THM}\label{thm:main}
	For every even integer $q\ge 6$, there exists a finite simple triangle-free
	graph $G_q$ with no $2$-factor such that
	\[
	\tau(G_q)=\frac{2q^2-q-2}{q^2+q}
	=2-\frac{3q+2}{q^2+q}.
	\]
\end{THM}

The remainder of this paper is organized as follows: in Section 2, we introduce some notation and preliminary 
results, and in Section 3, we prove Theorem~\ref{thm:main}. 

\section{Preliminary results}

Let $S$ and $T$ be disjoint sets of vertices of a graph $G$ 
and  $D$ be a component of $G-(S\cup T)$.
Then $D$ is said to be an \emph{odd component}
(resp.~\emph{even component})
if $e_G(D, T)\equiv 1\pmod{2}$
(resp.~$e_G(D, T)\equiv 0\pmod{2}$).
For each integer $k\ge 0$, we denote by $\CC_{2k+1}$ the set of all odd components $D$ of 
$G-(S\cup T)$ such that $e_G(D,T)=2k+1$. Let $\CC=\bigcup_{k\ge 0} \CC_{2k+1}$ and 
let $c(S, T)=|\CC|$.

Define 
$\delta(S, T)=2|S|+\sum_{y\in T} d_{G-S}(y)-2|T|-c(S, T)$.
It is easy to see $\delta(S, T)\equiv 0\pmod{2}$
for every $S$,~$T\subseteq V(G)$
with $S\cap T=\emptyset$.
We use the following criterion for the existence of a $2$-factor,
which is a special case of Tutte's $f$-Factor Theorem.
\begin{THM}[Tutte~\cite{Tutte1952}]\label{tutte's theorem}
	A graph $G$ has a $2$-factor if and only if
	$\delta(S, T)\ge 0$
	for every $S$,~$T\subseteq V(G)$
	with $S\cap T=\emptyset$.
\end{THM}

An ordered pair $(S,T)$ consisting of disjoint sets of vertices $S$ and $T$ in a graph $G$ 
is called a barrier if $\delta(S, T)\le -2$.
By Theorem~\ref{tutte's theorem},
if $G$ does not have a $2$-factor,
then $G$ has a barrier.

A conjecture of Chv\'atal~\cite{Chvatal1973} asserts that there is an
absolute constant $t_0$ such that every $t_0$-tough graph on at least
three vertices is \emph{pancyclic}, that is, contains a cycle of every
length from $3$ to its order. Bauer, van den Heuvel, and
Schmeichel~\cite{BauerVandenHeuvelSchmeichel1995} disproved this
conjecture by constructing, for every real number $t_0$, a
$t_0$-tough triangle-free graph. In the same year,
Alon~\cite{Alon1995} gave a different explicit construction based on
code graphs and established a stronger quantitative result, which we
state below.
 
\begin{THM}[Alon~{\cite[Theorem~3.1]{Alon1995}}]\label{lem:alon-code}
	For every integer $k>1$, there is a graph $G_k$ on $	n_k=2^{2k}$. 
	vertices that is $(2^k-1)$-regular and triangle-free with 
	\[
	\tau(G_k)
	>
	\frac{1}{3}
	\left(
	\frac{(2^k-1)^2}
	{(2\cdot 2^{k/2}+1)(2^k-1)
		+(2\cdot 2^{k/2}+1)^2}
	-1
	\right),
	\]
	and in particular  $\tau(G_k)=\Omega(2^{k/2})
	=\Omega(n_k^{1/4})$. 
\end{THM}

\begin{LEM}[Alon blocks]\label{lem:blocks}
	For every $M>0$ and every positive integer $r$, there is an
	even-order $d$-regular triangle-free graph $H$ such that
	\[
	\tau(H)\ge M
	\qquad\text{and}\qquad
	|V(H)|-d\ge r.
	\]
\end{LEM}

\begin{proof}
	Let $G_k$  be the code graph from Theorem~\ref{lem:alon-code}. Then $G_k$
	is $(2^k-1)$-regular with $|V(G_k)|=2^{2k}$, 
	and $G_k$ is triangle-free. Moreover, $\tau(G_k)=\Omega(2^{k/2})$, 
	so $\tau(G_k)\to\infty$ as $k\to\infty$. Also,
	$|V(G_k)|- (2^k-1) 
	=2^{2k}-2^k+1\to\infty$. 
	Hence, for all sufficiently large $k$, we have 
	$\tau(G_k)\ge M$ and $	|V(G_k)|-d \ge r$. 
	Taking $H=G_k$ proves the lemma. Finally, $|V(H)|=2^{2k}$ is even.
\end{proof}

We also use the standard observation that if a noncomplete graph is
$M$-tough, then its minimum degree is at least $2M$.  Indeed, any vertex cut
that disconnects the graph creates at least two components, so vertex
connectivity is at least $2M$, and minimum degree is at least vertex
connectivity.

\section{Proof of Theorem~\ref{thm:main}}

Fix an even integer $q\ge 6$, and put
\begin{equation}\label{eq:parameters}
	s=\frac q2-1,
	\qquad
	t=\binom q2,
	\qquad
	N=s+t+q,
	\qquad
	\tau=\frac{s+2t}{t+q}.
\end{equation}
Note that $\tau<2$.

Choose $M$ so large that
\begin{equation}\label{eq:Mchoice}
	\frac M2\ge \tau N.
\end{equation}
By Lemma~\ref{lem:blocks}, choose a $d$-regular triangle-free graph $H$ with
$\tau(H)\ge M$ and $|V(H)|-d\ge q-1$.
Fix $v\in V(H)$ and set $A=N_H(v)$.
Since $H$ is triangle-free, $A$ is independent. Also,
$|A|=d\ge 2M$, since $H$ is noncomplete and $M$-tough.
Choose a set $W$ of $q-1$ distinct vertices in $V(H)\setminus A$.

Take $q$ disjoint copies $H_1,\ldots,H_q$ of $H$. Write $A_i$ for the copy
of $A$ in $H_i$, and write
\[
W_i=\{w_{ij}:j\in[q]\setminus\{i\}\}
\]
for the copy of $W$ in $H_i$.
Add two independent sets
\[
S,\quad |S|=s,
\qquad\text{and}\qquad
T=\{t_{ij}:1\le i<j\le q\},\quad |T|=t.
\]
The additional edges are as follows.
\begin{enumerate}[(1)]
	\item Join $S$ completely to $T$.
	\item Join every vertex of $S$ to every vertex of every $A_i$.
	\item For each $1\le i<j\le q$, add precisely the two edges
	\[
	t_{ij}w_{ij},\qquad t_{ij}w_{ji}.
	\]
\end{enumerate}
There are no other edges between the displayed parts. Denote the resulting
graph by $G_q$. See Figure~\ref{fig:construction} for an illustration.

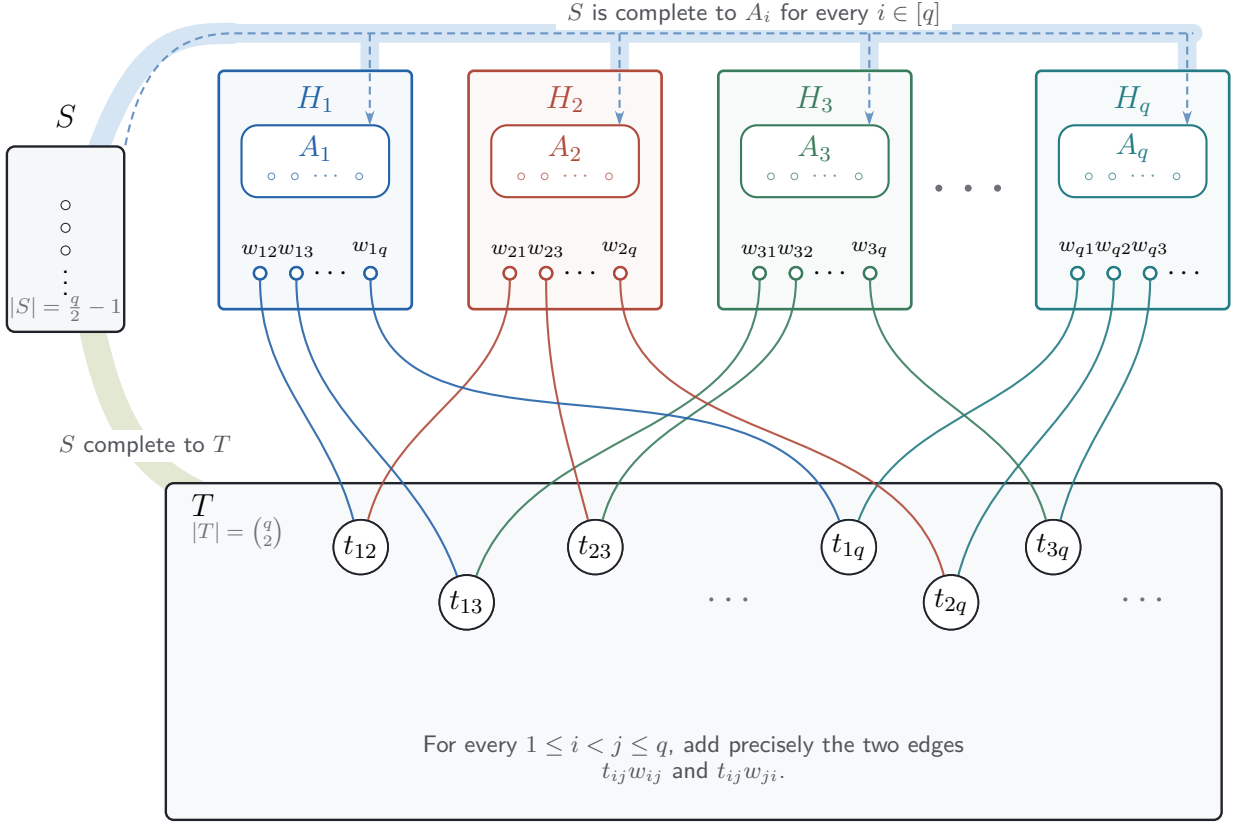
\begin{figure}[t]
\begin{center}
\begin{tikzpicture}[
	x=1cm,y=1cm,
	font=\sffamily,
	hblock/.style={draw=#1, line width=1pt, rounded corners=2pt,
		fill=#1!4, minimum width=2.55cm, minimum height=3.15cm},
	aset/.style={draw=#1, line width=.8pt, rounded corners=6pt,
		fill=white, minimum width=1.95cm, minimum height=.95cm},
	port/.style={circle, draw=#1, fill=white, line width=.9pt,
		minimum size=4.6pt, inner sep=0pt},
	tvertex/.style={circle, draw=ink, fill=white, line width=.7pt,
		minimum size=7.2mm, inner sep=1pt},
	explicit/.style={line width=.8pt, opacity=.9},
	note/.style={font=\sffamily\footnotesize, text=ink!78, align=center}
	]
	
	\coordinate (X1) at (4.35,9.15);
	\coordinate (X2) at (7.65,9.15);
	\coordinate (X3) at (10.95,9.15);
	\coordinate (Xq) at (15.15,9.15);
	
	\begin{scope}[on background layer]
		\draw[draw=joinA, line width=8pt, line cap=round]
		(1.35,9.30) .. controls (1.95,11.20) and (2.55,11.22) .. (3.25,11.22);
		\draw[draw=joinA, line width=7pt, line cap=round] (3.25,11.22)--(15.95,11.22);
		\foreach \x in {4.35,7.65,10.95,15.15}
		\draw[draw=joinA, line width=7pt, line cap=round]
		(\x+.72,11.22)--(\x+.72,10.00);
		\draw[draw=joinT, line width=10pt, line cap=round]
		(1.30,7.85) .. controls (1.60,6.05) and (2.05,5.38) .. (3.10,5.02);
	\end{scope}
	
	\node[draw=ink, rounded corners=3pt, fill=soft, minimum width=1.55cm,
	minimum height=2.45cm, line width=.9pt] (Sbox) at (1.05,8.50) {};
	\node[font=\bfseries\large, anchor=south] at ($(Sbox.north)+(0,.10)$) {$S$};
	\foreach \y in {8.95,8.65,8.35}
	\node[circle, draw=ink, fill=white, minimum size=3.4pt, inner sep=0pt] at (1.05,\y) {};
	\node at (1.05,8.03) {$\vdots$};
	\node[font=\scriptsize, text=ink!65] at (1.05,7.58) {$|S|=\frac q2-1$};
	
	\foreach \name/\sub/\x/\col in {H1/1/4.35/cOne,H2/2/7.65/cTwo,
		H3/3/10.95/cThree,Hq/q/15.15/cQ}{
		\node[hblock=\col] (\name) at (\x,9.15) {};
		\node[font=\bfseries\large, text=\col, anchor=north]
		at ($(\name.north)+(0,-.08)$) {$H_{\sub}$};
		\node[aset=\col] (A\sub) at ($(\name.center)+(0,.38)$) {};
		\node[font=\bfseries, text=\col] at ($(A\sub.center)+(0,.17)$) {$A_{\sub}$};
		\foreach \u in {-.58,-.27,.58}
		\node[circle, draw=\col!72, fill=white, minimum size=2.5pt, inner sep=0pt]
		at ($(A\sub.center)+(\u,-.20)$) {};
		\node[font=\scriptsize, text=\col!72]
		at ($(A\sub.center)+(.16,-.20)$) {$\cdots$};
	}
	\node[font=\bfseries\Huge, text=ink!65] at (13.05,9.15) {$\cdots$};
	
	\draw[draw=cOne!68, line width=.8pt, densely dashed]
	(Sbox.north east) .. controls (2.05,10.88) and (2.60,11.22) .. (3.25,11.22)
	-- (15.95,11.22);
	\foreach \x in {4.35,7.65,10.95,15.15}
	\draw[draw=cOne!68, line width=.8pt, densely dashed,
	-{Stealth[length=2.0mm,width=1.25mm]}]
	(\x+.72,11.22)--(\x+.72,10.00);
	\node[note, fill=white, inner xsep=4pt, inner ysep=1.5pt]
	at (10.15,11.48) {$S$ is complete to $A_i$ for every $i\in[q]$};
	
	\node[port=cOne] (w12) at (3.62,8.05) {};
	\node[port=cOne] (w13) at (4.10,8.05) {};
	\node at (4.58,8.05) {$\cdots$};
	\node[port=cOne] (w1q) at (5.08,8.05) {};
	\node[font=\scriptsize,anchor=south] at ($(w12)+(0,.08)$) {$w_{12}$};
	\node[font=\scriptsize,anchor=south] at ($(w13)+(0,.08)$) {$w_{13}$};
	\node[font=\scriptsize,anchor=south] at ($(w1q)+(0,.08)$) {$w_{1q}$};
	
	\node[port=cTwo] (w21) at (6.92,8.05) {};
	\node[port=cTwo] (w23) at (7.40,8.05) {};
	\node at (7.88,8.05) {$\cdots$};
	\node[port=cTwo] (w2q) at (8.38,8.05) {};
	\node[font=\scriptsize,anchor=south] at ($(w21)+(0,.08)$) {$w_{21}$};
	\node[font=\scriptsize,anchor=south] at ($(w23)+(0,.08)$) {$w_{23}$};
	\node[font=\scriptsize,anchor=south] at ($(w2q)+(0,.08)$) {$w_{2q}$};
	
	\node[port=cThree] (w31) at (10.22,8.05) {};
	\node[port=cThree] (w32) at (10.70,8.05) {};
	\node at (11.18,8.05) {$\cdots$};
	\node[port=cThree] (w3q) at (11.68,8.05) {};
	\node[font=\scriptsize,anchor=south] at ($(w31)+(0,.08)$) {$w_{31}$};
	\node[font=\scriptsize,anchor=south] at ($(w32)+(0,.08)$) {$w_{32}$};
	\node[font=\scriptsize,anchor=south] at ($(w3q)+(0,.08)$) {$w_{3q}$};
	
	\node[port=cQ] (wq1) at (14.42,8.05) {};
	\node[port=cQ] (wq2) at (14.90,8.05) {};
	\node[port=cQ] (wq3) at (15.38,8.05) {};
	\node at (15.86,8.05) {$\cdots$};
	\node[font=\scriptsize,anchor=south] at ($(wq1)+(0,.08)$) {$w_{q1}$};
	\node[font=\scriptsize,anchor=south] at ($(wq2)+(0,.08)$) {$w_{q2}$};
	\node[font=\scriptsize,anchor=south] at ($(wq3)+(0,.08)$) {$w_{q3}$};
	
	\node[draw=ink, rounded corners=3pt, fill=soft!80, line width=.9pt,
	minimum width=13.95cm, minimum height=4.45cm] (Tbox) at (9.35,3.05) {};
	\node[font=\bfseries\large, anchor=west] at ($(Tbox.north west)+(.22,-.30)$) {$T$};
	\node[font=\scriptsize, text=ink!65, anchor=west]
	at ($(Tbox.north west)+(.22,-.67)$) {$|T|=\binom q2$};
	
	\node[tvertex] (t12) at (4.95,4.43) {$t_{12}$};
	\node[tvertex] (t13) at (6.35,3.70) {$t_{13}$};
	\node[tvertex] (t23) at (8.05,4.43) {$t_{23}$};
	\node[font=\bfseries\Large, text=ink!65] at (9.85,3.72) {$\cdots$};
	\node[tvertex] (t1q) at (11.40,4.43) {$t_{1q}$};
	\node[tvertex] (t2q) at (12.75,3.70) {$t_{2q}$};
	\node[tvertex] (t3q) at (14.10,4.43) {$t_{3q}$};
	\node[font=\bfseries\Large, text=ink!65] at (15.32,3.72) {$\cdots$};
	
	\draw[explicit,draw=cOne]   (t12) to[out=105,in=-90] (w12);
	\draw[explicit,draw=cTwo]   (t12) to[out=75,in=-90]  (w21);
	\draw[explicit,draw=cOne]   (t13) to[out=110,in=-90] (w13);
	\draw[explicit,draw=cThree] (t13) to[out=70,in=-90]  (w31);
	\draw[explicit,draw=cTwo]   (t23) to[out=105,in=-90] (w23);
	\draw[explicit,draw=cThree] (t23) to[out=75,in=-90]  (w32);
	\draw[explicit,draw=cOne]   (t1q) to[out=110,in=-90] (w1q);
	\draw[explicit,draw=cQ]     (t1q) to[out=70,in=-90]  (wq1);
	\draw[explicit,draw=cTwo]   (t2q) to[out=105,in=-90] (w2q);
	\draw[explicit,draw=cQ]     (t2q) to[out=75,in=-90]  (wq2);
	\draw[explicit,draw=cThree] (t3q) to[out=105,in=-90] (w3q);
	\draw[explicit,draw=cQ]     (t3q) to[out=75,in=-90]  (wq3);
	
	\foreach \n/\col in {w12/cOne,w13/cOne,w1q/cOne,w21/cTwo,w23/cTwo,w2q/cTwo,
		w31/cThree,w32/cThree,w3q/cThree,wq1/cQ,wq2/cQ,wq3/cQ}
	\node[port=\col] at (\n) {};
	\foreach \n/\lab in {t12/{t_{12}},t13/{t_{13}},t23/{t_{23}},
		t1q/{t_{1q}},t2q/{t_{2q}},t3q/{t_{3q}}}
	\node[tvertex] at (\n) {$\lab$};
	
	\node[note, fill=white, inner sep=2pt] at (2.10,5.76)
	{$S$ complete to $T$};
	\node[note] at (9.35,1.43)
	{For every $1\le i<j\le q$, add precisely the two edges\\[-1pt]
		$t_{ij}w_{ij}$ and $t_{ij}w_{ji}$.\\[2pt]
		};
	
\end{tikzpicture}
\caption{Schematic of the construction.   Each $t_{ij}$ has exactly one edge
to $H_i$ and one edge to $H_j$.  The port sets lie outside the independent
neighborhoods $A_i$, which prevents triangles involving $S$.}
\label{fig:construction}
\end{center}
\end{figure}

\begin{CLA}\label{prop:trianglefree}
The graph $G_q$ is triangle-free.
\end{CLA}

\begin{proof}
Each $H_i$ is triangle-free, and both $S$ and $T$ are independent.  A triangle
containing two vertices of one block and a vertex of $S$ would use an edge
inside $A_i$, but $A_i$ is independent.  A triangle containing a vertex
$t_{ij}$ and two block vertices is impossible because $t_{ij}$ has only one
neighbor in each of $H_i$ and $H_j$.  Finally, a triangle using an edge between
$S$ and $T$ would require a port adjacent to $S$; the ports were chosen outside
$A_i$, and $A_i$ is exactly the set of block vertices adjacent to $S$.
\end{proof}

\begin{CLA}\label{prop:no-2-factor}
	The graph $G_q$ has no $2$-factor.
\end{CLA}

\begin{proof}
	We show that the ordered pair $(S,T)$ from the construction is a
	barrier. Every vertex $t_{ij}\in T$ has exactly two neighbors in
	$G_q-S$, namely $w_{ij}$ and $w_{ji}$. Hence
	\[
	\sum_{y\in T}d_{G_q-S}(y)=2|T|=2t.
	\]
	
	The components of $G_q-(S\cup T)$ are precisely
	$H_1,\ldots,H_q$. Each $H_i$ sends exactly $q-1$ edges to $T$,
	one through each vertex of $W_i$. Since $q$ is even, $q-1$ is odd,
	so all $q$ components are counted by $c(S,T)$. Thus $c(S,T)=q$.
	Therefore
	\[
	\delta(S,T)
	=2|S|+\sum_{y\in T}d_{G_q-S}(y)-2|T|-c(S,T)
	=2s+2t-2t-q=-2.
	\]
	Hence $(S,T)$ is a barrier, and Theorem~\ref{tutte's theorem}
	implies that $G_q$ has no $2$-factor.
\end{proof}

\begin{CLA}\label{prop:upper}
	$\tau(G_q)\le \tau$.
\end{CLA}

\begin{proof}
	Delete $W_0:=S\cup(\bigcup_{i=1}^q W_i)$ in $G_q$. This set has size
	$|W_0|=s+q(q-1)=s+2t$. Each vertex of $T$ becomes isolated.
	
	For every $i$, the graph $H_i-W_i$ is connected. Otherwise, the
	$M$-toughness of $H_i$ would give
	\[
	|W_i|\ge M\comp(H_i-W_i)\ge 2M,
	\]
	contrary to $|W_i|=q-1<2M$. Consequently,
	$\comp(G_q-W_0)=t+q$, and hence
	\[
	\tau(G_q)\le \frac{s+2t}{t+q}=\tau.
	\]
\end{proof}

\begin{CLA}\label{prop:lower}
	$\tau(G_q)\ge\tau$.
\end{CLA}

\begin{proof}
	Let $X\subseteq V(G_q)$ satisfy $\comp(G_q-X)\ge2$. Put
	\[
	U_i=X\cap V(H_i),
	\qquad
	h_i=\comp(H_i-U_i),
	\]
	where $h_i=0$ if $H_i-U_i$ is empty. Call $H_i$ \emph{bad} if
	\begin{enumerate}[(i)]
		\item $H_i-U_i$ is disconnected or empty; or
		\item $H_i-U_i$ is connected but $A_i\subseteq U_i$.
	\end{enumerate}
	Otherwise, call $H_i$ \emph{good}.
	
	For every bad block,
	\begin{equation}\label{eq:badcost}
		|U_i|\ge\frac M2(h_i+1).
	\end{equation}
	Indeed, if $h_i\ge2$, the $M$-toughness of $H_i$ gives
	$|U_i|\ge Mh_i$. If $h_i=0$, then $U_i=V(H_i)$ and
	$|U_i|\ge2M+1$. Finally, if $h_i=1$ and $H_i$ is bad, then
	$A_i\subseteq U_i$, so $|U_i|\ge|A_i|\ge2M$.
	
	Suppose first that there are $b\ge1$ bad blocks, and put
	$h=\sum_{H_i\text{ bad}}h_i$. Counting every surviving vertex of
	$S\cup T$ separately and every good block as one piece gives
	\[
	\comp(G_q-X)\le h+s+t+q=h+N.
	\]
	By \eqref{eq:badcost} and \eqref{eq:Mchoice},
	\[
	|X|
	\ge\frac M2(h+b)
	\ge\tau N(h+b)
	\ge\tau(h+N)
	\ge\tau\comp(G_q-X).
	\]
	Here $N(h+b)\ge h+N$ because $N\ge1$ and $b\ge1$.
	
	It remains to consider the case in which all blocks are good. Then
	each $H_i-U_i$ is nonempty and connected and contains a vertex of
	$A_i$. If a vertex of $S$ survived, it would join all $q$ surviving
	block pieces and all surviving vertices of $T$ into one component,
	contrary to $\comp(G_q-X)\ge2$. Therefore
	\begin{equation}\label{eq:Sdeleted}
		S\subseteq X.
	\end{equation}
	
	Let $R$ be the graph obtained from $G_q-X$ by contracting each
	connected subgraph $H_i-U_i$ to a single vertex, denoted by $i$.
	Since every block is good, each $H_i-U_i$ is nonempty; hence all
	$q$ branch vertices $1,\ldots,q$ occur in $R$. In particular, no
	branch vertex is deleted in this case.
	
	After deleting $S$, the only possible neighbors of $t_{ij}$ outside
	$T$ are $w_{ij}\in W_i$ and $w_{ji}\in W_j$. Thus, before the
	deletions belonging to $X$ are taken into account, the contracted
	structure is the subdivision $J_q$ of $K_q$, in which $t_{ij}$
	subdivides the edge $ij$. Deleting $t_{ij}$ removes the corresponding
	subdivision vertex, while deleting $w_{ij}$ removes the edge
	$it_{ij}$. A vertex of $U_i\setminus W_i$ contributes to $|X|$ but,
	since $H_i-U_i$ remains connected, removes no additional edge from
	$R$. Such deletions may therefore be disregarded when deriving a
	lower bound for $|X|$. Consequently, $R$ is obtained from $J_q$ by
	deleting some subdivision vertices and some branch--subdivision
	edges, but no branch vertices.
	
	Set
	\[
	x=\left|X\cap\left(T\cup\bigcup_{i=1}^qW_i\right)\right|.
	\]
	Thus $x$ counts the deleted subdivision vertices $t_{ij}$ and the
	deleted ports $w_{ij}$. Let $a$ be the number of components of $R$
	containing a branch vertex, and let $y$ be the number of isolated
	surviving subdivision vertices in $R$.
	
	Let $B_1,\ldots,B_a$ be the sets of branch vertices in these $a$
	components, and put $n_h=|B_h|$. Since all $q$ branch vertices
	survive, $\sum_{h=1}^a n_h=q$. The number of pairs of branch vertices
	belonging to different sets $B_h$ is
	\[
	C=\sum_{h<r}n_hn_r
	=\frac12\left(q^2-\sum_{h=1}^a n_h^2\right).
	\]
	For fixed $q$ and $a$, this quantity is minimized when
	$\sum_{h=1}^a n_h^2$ is maximized, which occurs for
	\[
	(n_1,\ldots,n_a)=(q-a+1,1,\ldots,1).
	\]
	For this partition,
	\begin{align*}
		C
		&=(q-a+1)(a-1)+\binom{a-1}{2}\\
		&=(a-1)\left(q-a+1+\frac{a-2}{2}\right)\\
		&=\frac{(a-1)(2q-a)}2.
	\end{align*}
	Consequently,
	\[
	C\ge C_a:=\frac{(a-1)(2q-a)}2.
	\]
	
	For every pair $i,j$ lying in different branch sets $B_h$ and
	$B_r$, at least one of $t_{ij},w_{ij},w_{ji}$ belongs to $X$;
	otherwise, the path $it_{ij}j$ would join the two corresponding
	components of $R$. Hence every such pair accounts for at least one
	of the $x$ deletions.
	
	Now let $t_{ij}$ be an isolated surviving subdivision vertex of $R$.
	Then both $w_{ij}$ and $w_{ji}$ belong to $X$. If $i$ and $j$ lie
	in different branch sets $B_h$ and $B_r$, one of these two deletions
	may already have been counted for the crossing pair $\{i,j\}$, but
	the other is an additional deletion. If $i$ and $j$ lie in the same
	branch set, the pair $\{i,j\}$ is not counted by $C$, so both
	deletions are additional. Since distinct subdivision vertices have
	distinct pairs of incident ports, it follows that
	\[
	x\ge C+y\ge C_a+y.
	\]
	Moreover, every isolated surviving subdivision vertex requires the
	deletion of both of its incident ports, and hence $x\ge2y$. Combining
	these bounds gives
	\begin{equation}\label{eq:reduced-cost}
		x\ge y+\max\{C_a,y\}.
	\end{equation}
	
	We next prove that, for $1\le a\le q$ and $0\le y\le t$,
	\begin{equation}\label{eq:reduced-ratio}
		\frac{s+y+\max\{C_a,y\}}{a+y}\ge\tau.
	\end{equation}
	For fixed $a$, extend the left-hand side to a function of a real
	variable $y$. It has the piecewise form
	\[
	F_a(y)=
	\begin{cases}
		\dfrac{s+C_a+y}{a+y},&0\le y\le C_a,\\[6pt]
		\dfrac{s+2y}{a+y},&C_a\le y\le t.
	\end{cases}
	\]
	On the interiors of these intervals,
	\[
	F_a'(y)=
	\begin{cases}
		\dfrac{a-s-C_a}{(a+y)^2},&0<y<C_a,\\[6pt]
		\dfrac{2a-s}{(a+y)^2},&C_a<y<t.
	\end{cases}
	\]
	The derivative has constant sign on each interval, so $F_a$ is
	monotone on each interval. It is therefore enough to check
	$y \in \{0,C_a,t\}$:
	\begin{equation}\label{eq:three-checks}
		\frac{s+C_a}{a}\ge\tau,\qquad
		\frac{s+2C_a}{a+C_a}\ge\tau,\qquad
		\frac{s+2t}{a+t}\ge\tau.
	\end{equation}
	
	For the first inequality, since $s=(q-2)/2$,
	\[
	C_a-(a-1)s
	=\frac{(a-1)(2q-a)}2-\frac{(a-1)(q-2)}2
	=\frac{(a-1)(q+2-a)}2\ge0.
	\]
	Therefore
	\[
	s+C_a
	=as+\bigl(C_a-(a-1)s\bigr)
	\ge as,
	\]
	and hence $(s+C_a)/a\ge s$. Since $q\ge6$, we have $s\ge2$,
	whereas $\tau<2$, so $(s+C_a)/a\ge s\ge2>\tau$.
	
	For the third inequality, $a\le q$ implies $a+t\le q+t$. Thus
	\[
	\frac{s+2t}{a+t}
	\ge\frac{s+2t}{q+t}
	=\tau.
	\]
	
	It remains to prove the middle inequality in
	\eqref{eq:three-checks}. Since all denominators are positive, it is
	equivalent to
	\[
	(t+q)(s+2C_a)\ge(s+2t)(a+C_a).
	\]
	Using
	\[
	s=\frac{q-2}{2},\qquad
	t=\frac{q(q-1)}2,\qquad
	C_a=\frac{(a-1)(2q-a)}2,
	\]
	the difference between the two sides is
	\begin{align*}
		&(t+q)(s+2C_a)-(s+2t)(a+C_a)\\
		&=\frac{q(q+1)}2
		\left(\frac{q-2}{2}+(a-1)(2q-a)\right)
		-\left(\frac{q-2}{2}+q(q-1)\right)
		\left(a+\frac{(a-1)(2q-a)}2\right)\\
		&=\frac14\Bigl(
		q(q+1)\bigl(q-2+2(a-1)(2q-a)\bigr)
		-(2q^2-q-2)\bigl(2a+(a-1)(2q-a)\bigr)
		\Bigr)\\
		&=\frac14\Bigl(
		q(q+1)\bigl(-2a^2+(4q+2)a-3q-2\bigr)
		-(2q^2-q-2)\bigl(-a^2+(2q+3)a-2q\bigr)
		\Bigr)\\
		&=\frac14\Bigl(
		-(3q+2)a^2+(2q^2+9q+6)a+q^3-7q^2-6q
		\Bigr)\\
		&=\frac14\Bigl(
		(q-a)(3q+2)a+(q-a)(q^2-7q-6)
		\Bigr)\\
		&=\frac{(q-a)\bigl((3q+2)a+q^2-7q-6\bigr)}4.
	\end{align*}
	The first factor is nonnegative because $a\le q$. For the second,
	$a\ge1$ gives
	\[
	(3q+2)a+q^2-7q-6
	\ge q^2-4q-4
	\ge0
	\]
	for $q\ge6$. This proves the middle inequality and hence
	\eqref{eq:reduced-ratio}.
	
	Contraction of connected subgraphs does not change the number of
	components. Moreover, subdivision vertices are pairwise nonadjacent,
	so every component of $R$ either contains a branch vertex or consists
	of a single isolated subdivision vertex. Therefore
	\[
	\comp(G_q-X)=\comp(R)=a+y.
	\]
	Finally, since $S\subseteq X$, by \eqref{eq:reduced-cost} and
	\eqref{eq:reduced-ratio},
	\[
	|X|
	\ge s+x
	\ge s+y+\max\{C_a,y\}
	\ge\tau(a+y)
	=\tau\comp(G_q-X).
	\]
	Thus every vertex cut $X$ satisfies
	$|X|\ge\tau\comp(G_q-X)$, proving that $\tau(G_q)\ge\tau$.
\end{proof}
	
Combining Claims~\ref{prop:upper} and~\ref{prop:lower} proves the equality
$\tau(G_q)=\tau$.  Together with Claim~\ref{prop:trianglefree} and~\ref{prop:no-2-factor}, this proves Theorem~\ref{thm:main}.
 
\section*{Declaration of Use of AI Tools}

During the preparation of this manuscript, the author  used ChatGPT 5.6 Plus 
to assist with language
editing, grammar, clarity, and formatting. The author  reviewed and verified
all AI-assisted edits and take full responsibility for the content of the
manuscript.

\end{document}